\documentclass[12pt]{amsart}
\usepackage{amssymb}
\usepackage{graphicx}
\usepackage{enumitem}

\theoremstyle{plain}
\newtheorem{theorem}{Theorem}[section]
\newtheorem{corollary}[theorem]{Corollary}
\newtheorem{lemma}[theorem]{Lemma}
\newtheorem{proposition}[theorem]{Proposition}

\theoremstyle{definition}
\newtheorem{definition}[theorem]{Definition}

\renewcommand{\phi}{\varphi}
\newcommand{\eps}{\varepsilon}

\newcommand{\kln}{{\mathcal K}_{\ell,N}}

\def\cM{\mathcal{M}}
\def\cMw{\widetilde{\mathcal{M}}}

\def\Z{{\mathbb Z}}
\def\cH{\mathcal{H}}

\begin{document}

\title{Topological entropy of generalized Bunimovich stadium billiards}

\author[M.~Misiurewicz]{Micha{\l}~Misiurewicz}

\address[Micha{\l}~Misiurewicz]
{Department of Mathematical Sciences\\ Indiana University-Pur\-due
University Indianapolis\\ 402 N. Blackford Street\\
Indianapolis, IN 46202\\ USA}
\email{mmisiure@math.iupui.edu}

\author[H.-K.~Zhang]{Hong-Kun~Zhang}

\address[Hong-Kun~Zhang]
{Department of Mathematics and Statistics\\ University of
Massachusetts\\ Amherst, MA 01003\\ USA}
\email{hongkun@math.umass.edu}



\date{December 8, 2022}

\thanks{Research of Micha{\l} Misiurewicz was partially
supported by grant number 426602 from the Simons Foundation.}

\begin{abstract}
We estimate from below the topological entropy of the generalized
Bunimovich stadium billiards. We do it for long billiard tables, and
find the limit of estimates as the length goes to infinity.
\end{abstract}

\maketitle

\section{Introduction}\label{sec-int}

In this paper, we generalize the results of~\cite{MZh} to a much
larger class of billiards. They are similar to Bunimovich stadium
billiards (see~\cite{B1}), but the semicircles are
replaced by almost arbitrary
curves. That is, those curves are not completely arbitrary, but the
assumptions on them is very mild. An example of such curves is shown
in Figure~\ref{fig0}

\begin{figure}[ht]
\begin{center}
\includegraphics[width=140truemm]{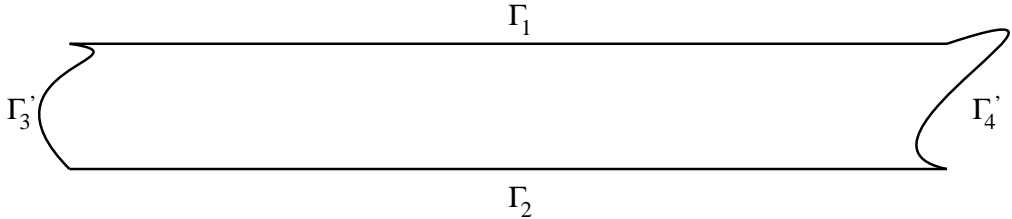}
\caption{Generalized Bunimovich stadium.}\label{fig0}
\end{center}
\end{figure}

We consider billiard maps (not the flows) for two-dimensional billiard
tables. Thus, the phase space of a billiard is the product of the
boundary of the billiard table and the interval $[-\pi/2,\pi/2]$ of
angles of reflection. This phase space will be denoted as $\cM$. We
will use the variables $(r,\phi)$, where $r$ parametrizes the table
boundary by the arc length, and $\phi$ is the angle of reflection.
Those billiards have the natural measure; it is
$c\cos\phi\;dr\;d\phi$, where $c$ is the normalizing constant. This
measure is invariant for the billiard map.

However, we will not be using this measure, but rather investigate our
system as a topological one. The first problem one encounters with
this approach is that the map can be discontinuous, or even not
defined at certain points. In particular, if we want to define
topological entropy of the system, we may use one of several methods,
but we cannot be sure that all of them will give the same result.

To go around this problem, similarly as in~\cite{MZh}, we consider a
compact subset of the phase space, invariant for the billiard map, on
which the map is continuous. Thus, the topological entropy of the
billiard map, no matter how defined, is larger than or equal to the
topological entropy of the map restricted to this subset.

Positive topological entropy is recognized as one of the forms of
chaos. In fact, topological entropy even measures how large this chaos
is. Hence, whenever we prove that the topological entropy is positive,
we can claim that the system is chaotic from the topological point of
view.

We will be using similar methods as in~\cite{MZh}. However, the class
of billiards to which our results can be applied, is much larger. The
class of Bunimovich stadium billiards, up to similarities, depends on
one positive parameter only. Our class is enormously larger, although
we keep the assumption that two parts of the billiard boundary are
parallel segments of straight lines. Nevertheless, some of our proofs
are simpler than those in~\cite{MZh}.

\section{Assumptions}\label{sec-ass}

We will think about the billiard table positioned as in
Figure~\ref{fig0}. Thus, we will use the terms \emph{horizontal,
  vertical, lower, upper, left, right}. While we are working with the
billiard map, we will also look at the billiard flow. Namely, we will
consider \emph{trajectory lines}, that is, line segments between two
consecutive reflections from the table boundary. For such a trajectory
line (we consider it really as a line, not a vector) we define its
\emph{argument} (as an argument of a complex number), which is the
angle between the trajectory line and a horizontal line. For
definiteness, we take the angle from $(-\pi/2,\pi/2]$. We will be also
  speaking about the arguments of lines in the plane. Moreover, for
  $x\in\cM$, we define the argument of $x$ as the argument of of the
  trajectory line joining $x$ with its image.

We will assume that the boundary of billiard table is the union of
four curves, $\Gamma_1$, $\Gamma_2$, $\Gamma_3'$ and $\Gamma_4'$. The
curves $\Gamma_1$ and $\Gamma_2$ are horizontal segments of straight
lines, and $\Gamma_2$ is obtained from $\Gamma_1$ by a vertical
translation. The curve $\Gamma_3'$ joins the left endpoints of
$\Gamma_1$ and $\Gamma_2$, while $\Gamma_4'$ joins the right endpoints
of $\Gamma_1$ and $\Gamma_2$ (see Figure~\ref{fig0}). We will consider
all four curves with endpoints, so they are compact.

\begin{definition} For $\eps\ge 0$, we will call a point $p\in\Gamma_i'$ ($i\in\{3,4\}$)
\emph{$\eps$-free} if any forward trajectory of the flow (here we mean
the full forward trajectory, not just the trajectory line), beginning
at $p$ with a trajectory line with argument whose absolute value is
less than or equal to $\eps$, does not collide with $\Gamma_i'$ before
it collides with $\Gamma_{7-i}'$.
Furthermore, we will call a subarc $\Gamma_i\subset\Gamma_i'$
\emph{$\eps$-free} (see Figure~\ref{fig2}) if:
\begin{enumerate}[label=(\alph*)]
\item $\Gamma_i$ is of class $C^1$;
\item Every point of $\Gamma_i$ is $\eps$-free;
\item There are points $p_{i+},p_{i-}\in\Gamma_i$ such that the
  argument of the line normal to $\Gamma_i$ is larger than or equal to
  $\eps$ at $p_{i+}$ and less than or equal to $-\eps$ at $p_{i-}$
  (see Figure~\ref{fig2});
\item $\Gamma_i$ is disjoint from $\Gamma_1\cup\Gamma_2$.
\end{enumerate}
Clearly, if $\Gamma_i$ is $\eps$-free then it is also $\delta$-free
for all $\delta\in(0,\eps)$.
\end{definition}
Our last assumption is that there is $\eps>0$ and $\eps$-free subarcs
$\Gamma_i\subset\Gamma_i'$ for $i=3,4$, such that
$\Gamma_3\cup\Gamma_4$ is disjoint from $\Gamma_1\cup\Gamma_2$. We
will denote the class of billiard tables satisfying all those
assumptions by $\cH(\eps)$.

\begin{figure}[ht]
\begin{center}
\includegraphics[width=140truemm]{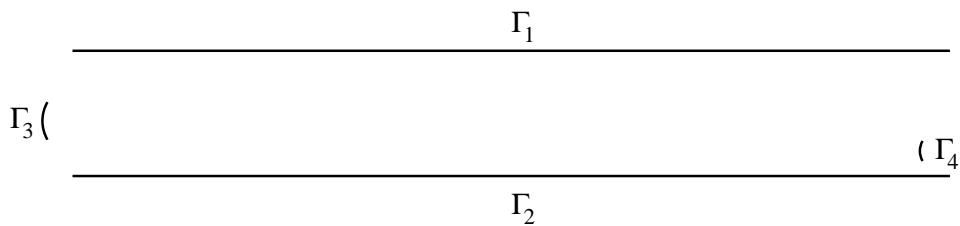}
\caption{Curves $\Gamma_i$, $i=1,2,3,4$.}\label{fig2}
\end{center}
\end{figure}

Observe that there are two simple situations when we know that there
is $\eps>0$ such that $\Gamma_i'$ has an $\eps$-free subarc. One is
when there is a $0$-free point $p_i\in\Gamma_i'$ such that there is a
neighborhood of $p_i$ where $\Gamma_i$ is of class $C^1$ and the
curvature of $\Gamma_i$ at $p_i$ exists and is non-zero (see
Figure~\ref{fig1}). The other one is when $\Gamma_i'$ is the graph of
a non-constant function $x=f(y)$ of class $C^1$ (then we take a
neighborhood of a point where $f$ attains its extremum; this
neighborhood may be large if the extremum is attained on an interval),
like $\Gamma_3'$ (but not $\Gamma_4'$) in Figure~\ref{fig0}.

\begin{figure}[ht]
\begin{center}
\includegraphics[width=140truemm]{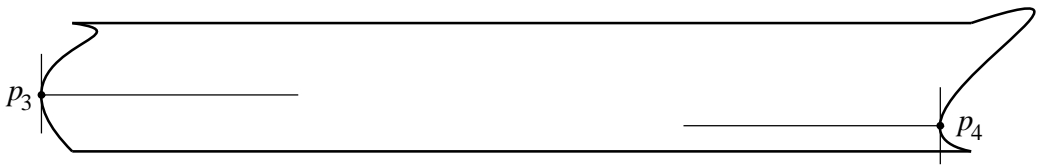}
\caption{Points $p_3$ and $p_4$.}\label{fig1}
\end{center}
\end{figure}

We forget about the other parts of the curves $\Gamma_i'$ and look
only at $\Gamma_i$, $i=1,2,3,4$ (see Figure~\ref{fig2}).

Let us mention that since we will be using only those four pieces of
the boundary of the billiard table, it does not matter whether the
rest of the boundary is smooth or not. If it is not smooth, we can
include it (times $[-\pi/2,\pi/2]$) into the set of singular points,
where the billiard map is not defined.

\section{Coding}\label{sec-cod}

We consider a billiard table from the class $\cH(\eps)$. Since
transforming the table by homothety does not change the entropy, we
may assume that the distance between $\Gamma_1$ and $\Gamma_2$ is 1.
Now we can introduce a new characteristic of our billiard table. We
will say that a billiard table from the class $\cH(\eps)$ is in the
class $\cH(\eps,\ell)$ if the horizontal distance between $\Gamma_3$
and $\Gamma_4$ is at least $\ell$. We can think about $\ell$ as a big
number (it will go to infinity).

We start with a trivial geometrical fact, that follows immediately
from the rule of reflection. We include the assumption
  that the absolute values of the arguments are smaller than $\pi/6$
  in order to be sure that the absolute value of the argument of $T_2$
  is smaller than $\pi/2$.

\begin{lemma}\label{l-trivial}
If $T_1$ and $T_2$ are incoming and outgoing parts of a trajectory
reflecting at $q$ and the argument of the line normal to the boundary
of the billiard at $q$ is $\alpha$, and $|\alpha|,|\arg(T_1)|<\pi/6$,
then $\arg(T_2)=2\alpha-\arg(T_1)$.
\end{lemma}

We consider only trajectories that reflect from the curves $\Gamma_i$,
$i=1,2,3,4$. In order to have control over this subsystem, we fix an
integer $N>1$ and denote by $\kln$ the space of points whose
(discrete) trajectories go only through $\Gamma_i$, $i=1,2,3,4$ and
have no $N+1$ consecutive collisions with the straight segments.

We can unfold the billiard table by using reflections from the
straight segments (see Figure~\ref{fig3}). The liftings of
trajectories (of the flow) consist of segments between points of
liftings of $\Gamma_3$ and $\Gamma_4$. In $\kln$ they go at most $N$
levels up or down.

\begin{figure}[ht]
\begin{center}
\includegraphics[width=140truemm]{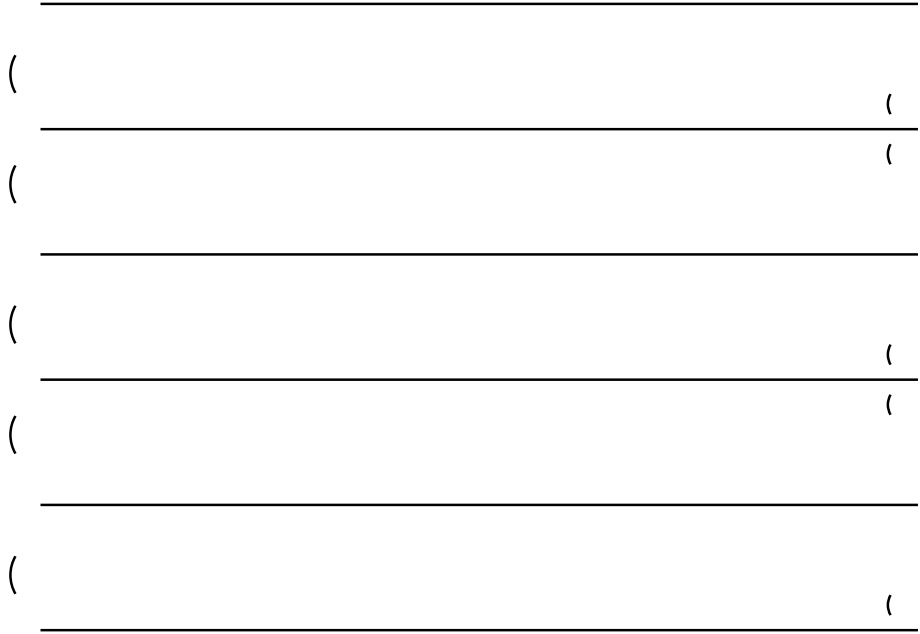}
\caption{Five levels of the unfolding. Only $\Gamma_3$ and $\Gamma_4$
  are shown instead of $\Gamma_3'$ and $\Gamma_4'$.}\label{fig3}
\end{center}
\end{figure}

Now for a moment we start working on the lifted billiard. That is, we
consider only $\Gamma_3$ and $\Gamma_4$, but at all levels, as pieces
of the boundary from which the trajectories of the flow can reflect.
We denote those pieces by $\Gamma_{i,k}$, where $i\in\{3,4\}$ and
$k\in\Z$. Clearly, flow trajectories from some points $(r,\phi)$ will
not have more collisions, so the lifted billiard map $F$ will be not
defined at such points. We denote by $\cMw$ the product of the union
of all sets $\Gamma_{i,k}$ and the interval $[\pi/2,\pi/2]$.

Now we specify how large $\ell$ should be for given $N,\eps$ in order
to get nice properties of the billiard map restricted to $\kln$.

Assume that our billiard table belongs to $\cH(\eps,\ell)$ and fix
$i\in\{3,4\}$, $k\in\Z$. Call a continuous map $\gamma:[a,b]\to\cMw$,
given by
\[
\gamma(t)=(\gamma_r(t),\gamma_\phi(t)),
\]
an \emph{$(i,k,\eps)$-curve} if $\gamma_r([a,b])=\Gamma_{i,k}$ and for
every $t\in[a,b]$ the absolute value of the argument of
the trajectory line \emph{incoming to} $\gamma(t)$ is
at most $\eps$. We can think of $\gamma$ as a bundle of trajectories
of a flow incoming to $\Gamma_{i,k}$. In order to be able to use
Lemma~\ref{l-trivial}, we will always assume that $\eps<\pi/6$.

\begin{lemma}\label{angles}
Assume that the billiard table belongs to $\cH(\eps,\ell)$ and fix
$N\ge 0$, $i\in\{3,4\}$, $k\in\Z$, and $j\in\{-N,-N+1,\dots,N-1,N\}$.
Assume that
\begin{equation}\label{e-1}
\ell\ge\frac{N+1}{\tan\eps}
\end{equation}
Then every $(i,k,\eps)$-curve $\gamma$ has a subcurve whose image
under $F$ (that is, $F\circ\gamma|_{[a',b']}$ for some subinterval
$[a',b']\subset[a,b]$) is a $(7-i,k+j,\eps)$-curve.
\end{lemma}

\begin{proof}
There are points $c_-,c_+\in[a,b]$ such that $\gamma_r(c_-)$ is a
lifting of $p_{i-}$ and $\gamma_r(c_+)$ is a lifting of $p_{i+}$.
Then, by Lemma~\ref{l-trivial}, the lifted trajectory line outgoing
from $\gamma(c_-)$ (respectively, $\gamma(c_+)$) has argument smaller
than $-\eps$ (respectively, larger than $\eps$).
Since the direction of the line normal to $\Gamma_{i,k}$
  at the point $\gamma_r(t)$ varies continuously with $t$, the
  argument of the lifted trajectory line outgoing from $\gamma(t)$
  also varies continuously with $t$. Therefore, there is a subinterval
  $[a'',b'']\subset [a,b]$ such that at one of the points $a'',b''$
  this argument is $-\eps$, at the other one is $\eps$, and in between
  is in $[-\eps,\eps]$. When the bundle of lifted trajectory lines
  starting at $\gamma([a'',b''])$ reaches liftings of $\Gamma_{7-i}$,
  it collides with all points of $\Gamma_{7-i,k+j}$ whenever
  $j+1\le\ell\tan\eps$. By~\eqref{e-1}, this includes all $j$ with
  ${j}\le N$. Therefore, there is a subinterval
  $[a',b']\subset[a'',b'']$ such that
  $(F\circ\gamma)_r([a',b'])=\Gamma_{7-i,k+j}$. The arguments of the
  lifted trajectory lines incoming to $(F\circ\gamma)([a',b'])$ are in
  $[-\eps,\eps]$, so we get a $(7-i,k+j,\eps)$-curve.
\end{proof}

Using this lemma inductively we get immediately the next lemma.

\begin{lemma}\label{l-iti}
Assume that the billiard table belongs to $\cH(\eps,\ell)$ and fix
$N\ge 0$ such that~\eqref{e-1} is satisfied. Then for every finite
sequence
\[
(k_{-j},\dots,k_{-1},k_0,k_1,\dots,k_j)
\]
of integers with absolute values at most $N$ there is a trajectory
piece in the lifted billiard going between lifting of  $\Gamma_3$ and
$\Gamma_4$ with the differences of levels
$k_{-j},\dots,k_{-1},k_0,k_1,\dots,k_j$.
\end{lemma}

Note that in the above lemma we are talking about trajectory pieces of
length $2j+1$, without requiring that those pieces can be extended
backward or forward to a full trajectory.

\begin{proposition}\label{p-iti}
Under the assumption of Lemma~\ref{l-iti}, for every two-sided
sequence
\[
(\dots,k_{-2},k_{-1},k_0,k_1,k_2,\dots)
\]
of integers with absolute values at most $N$ there is a trajectory in
the lifted billiard going between liftings of $\Gamma_3$ and
$\Gamma_4$ with the differences of levels
$\dots,k_{-2},k_{-1},k_0,k_1,k_2,\dots$.
\end{proposition}

\begin{proof}
For every finite sequence $(k_{-j},\dots,k_{-1},k_0,k_1,\dots,k_j)$
the set of points of $\Gamma_3\times[-\pi/2,\pi/2]$ or
$\Gamma_4\times[-\pi/2,\pi/2]$ whose trajectories from time $-j$ to
$j$ exist  and satisfy Lemma~\ref{l-iti} is compact and nonempty. As
$j$ goes to infinity, we get a nested sequence of compact sets. Its
intersection is the set of points whose trajectories behave in the way
we demand, and it is nonempty.
\end{proof}

Consider the following subshift of finite type
$(\Sigma_{\ell,N},\sigma)$. The states are
\[
-N,-N+1,\dots,-1,0,1,\dots,N-1,N,
\]
and the transitions are: from 0 to 0, 1 and $-1$, from $i$ to $i+1$
and 0 if $1\le i\le N-1$, from $N$ to 0, from $-i$ to $-i-1$ and 0 if
$1\le i\le N-1$, and from $-N$ to 0. Each trajectory of a point from
$\kln$ can be coded by assigning the symbol 0 to
$\Gamma_3\cup\Gamma_4$ and for the parts between two zeros either
$1,2,\dots,j$ if the the first point is in $\Gamma_1$, or
$-1,-2,\dots,-j$ if the first point is in $\Gamma_2$. This defines a
map from $\kln$ to $\Sigma_{\ell,N}$. This map is continuous, because
the preimage of every cylinder is open (this follows immediately from
the fact that the straight pieces of our trajectories of the billiard
flow intersect the arcs $\Gamma_i$, $i=1,2,3,4$, only at the endpoints
of those pieces, and that the arcs are disjoint). It is a surjection
by Proposition~\ref{p-iti}. Therefore it is a semiconjugacy, and
therefore, the topological entropy of the billiard map restricted to
$\kln$ is larger than or equal to the topological entropy of
$(\Sigma_{\ell,N},\sigma)$.

\section{Computation of topological entropy}\label{sec-cote}

In the preceding section we obtained a subshift of finite type. Now
we have to compute its topological entropy. If the alphabet
of a subshift of finite type is $\{1,2,\dots,n\}$, then we
can write the \emph{adjacency matrix} $M=(m_{ij})_{i,j=1}^n$, where
$m_{ij}=1$ if there is a transition from $i$ to $j$ and $m_{ij}=0$
otherwise. Then the topological entropy of our subshift is the
logarithm of the spectral radius of $M$ (see~\cite{K, ALM}).

In the case of large, but not too complicated, matrices, in order to
compute the spectral radius one can use the \emph{rome method}
(see~\cite{BGMY, ALM}). For the adjacency matrices of
$(\Sigma_{\ell,N},\sigma)$ this method is especially simple. Namely,
if we look at the paths given by transitions, we see that 0 is a rome:
all paths lead to it. Then we only have to identify the lengths of all
paths from 0 to 0 that do not go through 0 except at the beginning and
the end. The spectral radius of the adjacency matrix is then the
largest zero of the function $\sum x^{-p_i}-1$, where the sum is over
all such paths and $p_i$ is the length of the $i$-th path.

\begin{lemma}\label{l-ent}
Topological entropy of the system $(\Sigma_{\ell,N},\sigma)$ is the
logarithm of the largest root of the equation
\begin{equation}\label{eq0}
x^2-2x-1=-2x^{-N}.
\end{equation}
\end{lemma}

\begin{proof}
The paths that we mentioned before the lemma, are: one path of length
1 (from 0 directly to itself), and two paths of length $2,3,\dots,N+1$
each. Therefore, our entropy is the logarithm of the largest zero of
the function $2(x^{-(N+1)}+\dots+x^{-3}+x^{-2})+x^{-1}-1$. We have
\[
x(1-x)\big(2(x^{-(N+1)}+\dots +x^{-3}+x^{-2})+x^{-1}-1\big)=
(x^2-2x-1)+2x^{-N},
\]
so our entropy is the logarithm of the largest root of
equation~\eqref{eq0}.
\end{proof}

\begin{corollary}\label{c-ent}
Assume that the billiard table belongs to $\cH(\eps,\ell)$ and fix
$N\ge 0$ such that~\eqref{e-1} is satisfied. Then the topological
entropy of the billiard map restricted to $\kln$ is larger than or
equal to the logarithm of the largest root of equation~\eqref{eq0}.
\end{corollary}

A particular case of this corollary gives us a sufficient condition
for positive topological entropy. Namely, notice that the largest root
of the equation $x^2-2x-1=-2x^{-1}$ is $2$.

\begin{corollary}\label{c-ent1}
Assume that the billiard table belongs to $\cH(\eps,\ell)$ and
$\ell\tan\eps\ge 2$. Then the topological entropy of the billiard map
is at least $\log2$, so the map is chaotic in topological sense.
\end{corollary}

{
It is interesting how this estimate works for the classical Bunimovich
stadium billiard. In fact, for the estimate we will improve a little comparing to 
the above Corollary.

\begin{figure}[ht]
\begin{center}
\includegraphics[width=50truemm]{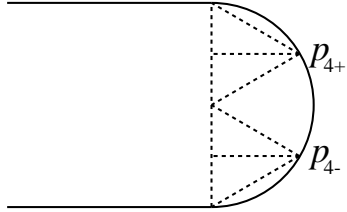}
\caption{Computations for the stadium billiard.}\label{fig10}
\end{center}
\end{figure}

\begin{proposition}\label{st}
If the rectangular part of a stadium has the length/width ratio larger
than $\sqrt3\approx 1.732$ (see Figure~\ref{fig4}), the billiard map
has topological entropy at least $\log 2$.
\end{proposition}

\begin{proof}
We can take $\eps$ as close to $\pi/6$ as we want (see
Figure~\ref{fig10}), so we get the assumption in the corollary
$\ell>2\sqrt3$. However, the factor 2 (in general, $N+1$
in~\eqref{e-1}) was taken to get an estimate that works for all
possible choices of $\Gamma_i$, $i=3,4$. For our concrete choice it is
possible to replace it by the vertical size of
$\Gamma_{i,0}\cup\Gamma_{i,1}$ (or $\Gamma_{i,0}\cup\Gamma_{i,-1}$,
bit it is the same in our case). This number is not 2, but $\frac32$.
Thus, we really get $\ell>\frac32\sqrt3$. If $\ell'$ is the length of
the rectangular part of the stadium, then
$\ell=\ell'+2\cdot\frac{\sqrt3}4=\ell'+\frac12\sqrt3$.
This gives us $\ell'>\sqrt3$.
\end{proof}
}

\begin{figure}[ht]
\begin{center}
\includegraphics[width=140truemm]{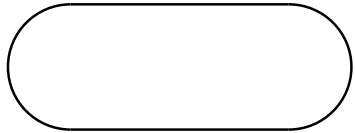}
\caption{Stadium billiard with topological entropy at least
  $\log2$.}\label{fig4}
\end{center}
\end{figure}

Now we can prove the main result of this paper.

\begin{theorem}\label{main}
For the billiard tables from the class $\cH$
with the shapes of $\Gamma_3$ and $\Gamma_4$ fixed, the lower limit of
the topological entropy of the generalized Bunimovich stadium
billiard, as its length $\ell$ goes to infinity, is at least
$\log(1+\sqrt2)$.
\end{theorem}

\begin{proof}
In view of Corollary~\ref{c-ent} and the fact that the largest root
of the equation $x^2-2x-1=0$ is $1+\sqrt2$, we only have to prove that
the largest root of the equation~\eqref{eq0} converges to the largest
root of the equation $x^2-2x-1=0$ as $N\to\infty$. However, this
follows from the fact that in the neighborhood of $1+\sqrt2$ the
right-hand side of~\eqref{eq0} goes uniformly to 0 as $N\to\infty$.
\end{proof}

\section{Generalized semistadium billiards}

In a similar way we can investigate generalized semistadium billiards.
They are like generalized stadium billiards, but one of the caps
$\Gamma_3',\Gamma_4'$ is a vertical straight line segment. The other
one contains an $\eps$-free subarc. This class contains, in
particular, Bunimovich's Mushroom billiards (see~\cite{B2}), see
Figure~\ref{fig6}. We will be talking about the classes $\cH_{1/2}$,
$\cH_{1/2}(\eps)$ and $\cH_{1/2}(\eps,\ell)$ of billiard tables.

\begin{figure}[ht]
\begin{center}
\includegraphics[width=140truemm]{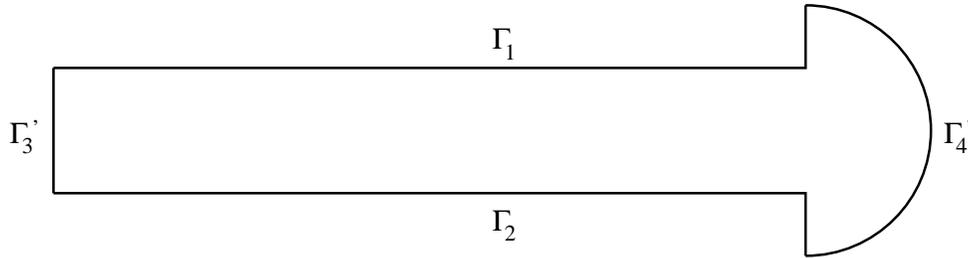}
\caption{A mushroom.}\label{fig6}
\end{center}
\end{figure}

When we construct a lifting, we add the reflection from the flat
vertical cap. In such a way we obtain the same picture as in
Section~\ref{sec-cod}, except that there is an additional vertical
line through the middle of the picture, and we have to count the flow
trajectory crossing it as an additional reflection (see
Figure~\ref{fig7}). Note that since we will be working with the lifted
billiard, in the computations we can take $2\ell$ instead of $\ell$.
In particular, inequality~\eqref{e-1} will be now replaced by
\begin{equation}\label{e-2}
\ell\ge\frac{N+1}{2\tan\eps}
\end{equation}

\begin{figure}[ht]
\begin{center}
\includegraphics[width=140truemm]{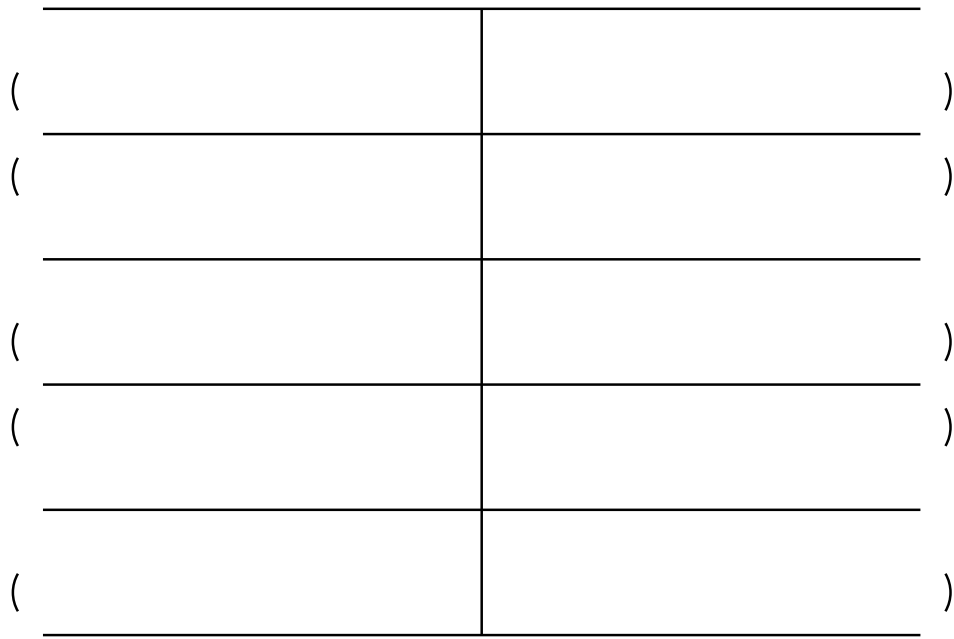}
\caption{Unfolding.}\label{fig7}
\end{center}
\end{figure}

Computation of the topological entropy is this time a little more
complicated. We cannot claim that after coding we are obtaining a
subshift of finite type. This is due to the fact that if $\Gamma_i'$
is a vertical segment, we would have to take $\Gamma_i=\Gamma_i'$, and
$\Gamma_i$ would not be disjoint from $\Gamma_1$ and $\Gamma_2$. The
second reason is that the moment when the reflection from the vertical
segment occurs depends on the argument of the trajectory line.

The formula for the topological entropy of the subshift of finite type
comes from counting of number of cylinders of length $n$ and then
taking the exponential growth rate of this number as $n$ goes to
infinity. Here we can try do exactly the same, but the problem occurs
with the growth rate, since we have additional reflections from the
vertical segment. This means that the cylinders of length $n$ from
Section~\ref{sec-cod} correspond not to time $n$, but to some
larger time. How much larger, depends on the cylinder. However, there
cannot be two consecutive reflections from the vertical segment, so
this time is not larger than $2n$, and by extending the trajectory we
may assume that it is equal to $2n$ (maybe there will be more
cylinders, but we need only a lower estimate). Thus, if the number of
cylinders (which we count in Section~\ref{sec-cod}) of length $n$ is
$a_n$, instead of taking the limit of $\frac1n\log a_n$ we take the
limit of $\frac1{2n}\log a_n$, that is, the half of the limit from
Section~\ref{sec-cod}. In such a way we get the following results.

\begin{proposition}\label{p-ent}
Assume that the billiard table belongs to $\cH_{1/2}(\eps,\ell)$ and
fix $N\ge 0$ such that~\eqref{e-2} is satisfied. Then the topological
entropy of the billiard map restricted to $\kln$ is larger than or
equal to one half of the logarithm of the largest root of
equation~\eqref{eq0}.
\end{proposition}

\begin{proposition}\label{p-ent1}
Assume that the billiard table belongs to $\cH_{1/2}(\eps,\ell)$ and
$\ell\tan\eps\ge 1$. Then the topological entropy of the billiard map
is at least $\frac12\log2$, so the map is chaotic in topological
sense.
\end{proposition}

\begin{theorem}\label{mainn}
For the billiard tables from the class $\cH_{1/2}$ with the shape of
$\Gamma_3$ or $\Gamma_4$ (the one that is not the vertical segment)
fixed, the lower limit of the topological entropy of the generalized
Bunimovich stadium billiard, as its length $\ell$ goes to infinity, is
at least $\frac12\log(1+\sqrt2)$.
\end{theorem}

{
We can apply Proposition~\ref{p-ent1} to the Bunimovich mushroom
billiard in order to get entropy at least $\frac12\log2$. As for the
stadium, we need to make some computations, and again, we will make a
slight improvement in the estimates. The interior of the
mushroom billiard consist of a rectangle (the stalk) and a half-disk
(the cap). According to our notation, the stalk is of vertical size 1;
denote its horizontal size by $\ell'$. Moreover, denote the radius of
the cap by $t$.

\begin{proposition}\label{mu}
If $\ell'>\frac12\sqrt{16t^2-1}$ then the topological entropy of the
mushroom billiard is at least $\frac12\log2$.
\end{proposition}

\begin{proof}
Look at Figure~\ref{fig9}, where the largest possible $\eps$ is used.
We have $t\sin\eps=1/4$.

\begin{figure}[ht]
\begin{center}
\includegraphics[width=60truemm]{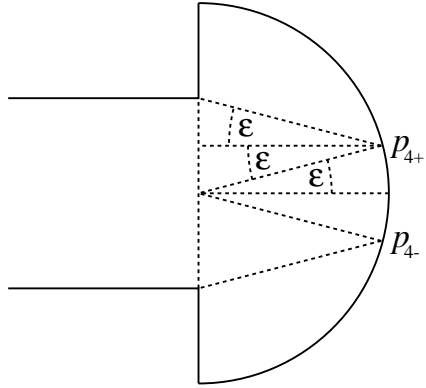}
\caption{Computations for a mushroom.}\label{fig9}
\end{center}
\end{figure}

Therefore, $\tan\eps=1/\sqrt{16t^2-1}$. Similarly as for the stadium,
when we use~\eqref{e-2} with $N=1$, we may replace $N+1$ by $\frac32$.
Taking into account that we need a strict inequality, we get
$\ell>\frac34\sqrt{16t^2-1}$. However,
$\ell=\ell'+t\cos\eps=\ell'+\frac14\sqrt{16t^2-1}$, so our condition
is $\ell'>\frac12\sqrt{16t^2-1}$.
\end{proof}

\begin{figure}[ht]
\begin{center}
\includegraphics[width=140truemm]{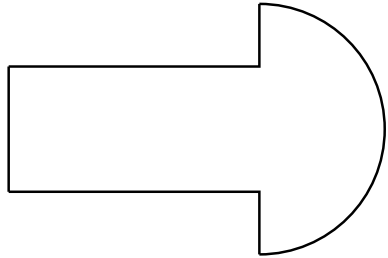}
\caption{A mushroom with topological entropy at least
  $\frac12\log2$.}\label{fig8}
\end{center}
\end{figure}

Observe that the assumption of Proposition~\ref{mu} is satisfied if
the length of the stalk is equal to or larger than the diameter of the
cap (see Figure~\ref{fig8}).
}

\end{document}